\documentclass[10pt]{amsart}
\usepackage{amssymb,a4wide}
\usepackage{amsthm}
\usepackage{graphicx}
\usepackage[T2A]{fontenc}
\usepackage[latin1]{inputenc}

\newcommand{\weg}[1]{}

\theoremstyle{plain}
\newtheorem{thm}{Theorem}
\newtheorem*{thm*}{Theorem}
\newtheorem{lem}{Lemma}

\theoremstyle{definition}
\newtheorem{defn}{Definition}

\newtheorem{exmp}{Example}
\theoremstyle{remark}
\newtheorem{rem}{Remark}

\title[Two remarks on $PQ^{\epsilon}$-projectivity of Riemannian metrics]{Two remarks on $PQ^{\epsilon}$-projectivity of Riemannian metrics}
  \author{Vladimir S. Matveev  and  Stefan Rosemann}
\thanks{Institute of Mathematics, FSU Jena, 07737 Jena Germany,\\  vladimir.matveev@uni-jena.de, stefan.rosemann@uni-jena.de}
\thanks{partially supported by  GK 1523 of DFG}
\begin{document}

\begin{abstract}
We show that $PQ^{\epsilon}$-projectivity of two Riemannian metrics introduced in \cite{Top2003} implies affine equivalence of the metrics unless $\epsilon\in\{0,-1,-3,-5,-7,...\}$. Moreover, we show that for $\epsilon=0$, $PQ^{\epsilon}$-projectivity implies projective equivalence.
\end{abstract}

\maketitle

\section{Introduction}
\subsection{$PQ^{\epsilon}$-projectivity of Riemannian metrics}
Let $g,\bar{g}$ be two Riemannian metrics on an $m$-dimensional manifold $M$. Consider $(1,1)$-tensors $P,Q$ which satisfy 
\begin{align}
\begin{array}{c}
g(P.,.)=-g(.,P.),\,\,\,g(Q.,.)=-g(.,Q.)\vspace{1mm}\\
\bar{g}(P.,.)=-\bar{g}(.,P.),\,\,\,\bar{g}(Q.,.)=-\bar{g}(.,Q.)\vspace{1mm}\\
PQ=\epsilon Id,\end{array}\label{eq:conditions}
\end{align}
where $Id$ is the identity on $TM$ and $\epsilon$ is a real number, $\epsilon\neq 1, m+1$. 
The following definition was introduced in \cite{Top2003}.
\begin{defn}
The metrics $g,\bar{g}$ are called $PQ^{\epsilon}$-projective if for a certain $1$-form $\Phi$ the Levi-Civita connections $\nabla$ and $\bar{\nabla}$ of $g$ and $\bar{g}$ satisfy
 \begin{align}
\bar{\nabla}_{X}Y-\nabla_{X}Y=\Phi(X)Y+\Phi(Y)X-\Phi(PX)QY-\Phi(PY)QX\label{eq:pqproj}
\end{align}
for all vector fields $X,Y$.
\end{defn}
\begin{exmp}
If the two metrics $g$ and $\bar{g}$ are \emph{affinely equivalent}, i.e. $\nabla=\bar{\nabla}$, then they are $PQ^{\epsilon}$-projective with $P,Q,\epsilon$ arbitrary and $\Phi\equiv 0$. 
\end{exmp}
\begin{exmp}\label{exmp:projective}
Suppose that $\Phi(P.)=0$ or $Q=0$ and $\epsilon=0$. It follows that equation \eqref{eq:pqproj} becomes 
 \begin{align}
\bar{\nabla}_{X}Y-\nabla_{X}Y=\Phi(X)Y+\Phi(Y)X.\label{eq:projective}
\end{align}
By Levi-Civita \cite{Levi1896}, equation \eqref{eq:projective} is equivalent to the condition that $g$ and $\bar{g}$ have the same geodesics considered as unparametrized curves, i.e., that $g$ and $\bar{g}$ are \emph{projectively equivalent}. The theory of projectively equivalent metrics has a very long tradition in differential geometry, see for example \cite{Sinjukov,Mikes1996,MT2001,Inventiones2003} and the references therein.
\end{exmp}
\begin{exmp}\label{exmp:hprojective}
Suppose that $P=Q=J$ and $\epsilon=-1$. It follows that $J$ is an almost complex structure, i.e., $J^{2}=-Id$, and by \eqref{eq:conditions} the metrics $g$ and $\bar{g}$ are required to be hermitian with respect to $J$. Equation \eqref{eq:pqproj} now reads
 \begin{align}
\bar{\nabla}_{X}Y-\nabla_{X}Y=\Phi(X)Y+\Phi(Y)X-\Phi(JX)JY-\Phi(JY)JX.\label{eq:hprojective}
\end{align}
This equation defines the \emph{$h$-projective equivalence} of the hermitian metrics $g$ and $\bar{g}$ and was introduced for the first time by Otsuki and Tashiro in \cite{Otsuki1954,Tashiro1956} for Kählerian metrics. The theory of $h$-projectively equivalent metrics was introduced as an analog of projective geometry in the Kählerian situation and has been studied actively over the years, see for example \cite{Mikes,Kiyohara2010,ApostolovI,MatRos} and the references therein. 
\end{exmp}

\subsection{Results}
The aim of our paper is to give a proof of the following two theorems:
\begin{thm}\label{thm:result1}
Let Riemannian metrics $g$ and $\bar{g}$ be $PQ^{\epsilon}$-projective. If $g$ and $\bar{g}$ are not affinely equivalent, the number $\epsilon$ is either zero or an odd negative integer, i.e., $\epsilon\in\{0,-1,-3,-5,-7,...\}$. 
\end{thm}
\begin{thm}\label{thm:result2}
Let Riemannian metrics $g$ and $\bar{g}$ be $PQ^{\epsilon}$-projective. If $\epsilon=0$ then $g$ and $\bar{g}$ are projectively equivalent.
\end{thm}

\subsection{Motivation and open questions}
As it was shown in \cite{Top2003}, $PQ^{\epsilon}$-projectivity of the metrics $g,\bar{g}$ allows us to construct a family of commuting integrals for the geodesic flow of $g$ (see Theorem \ref{thm:integrals} and equation \eqref{eq:integrals} below). The existence of these integrals is an interesting phenomenon on its own. Besides, it appeared to be a powerful tool in the study of projectively equivalent and $h$-projectively equivalent metrics (Examples \ref{exmp:projective},\ref{exmp:hprojective}), see \cite{Kiyohara2010,Inventiones2003,MT2001,MatRos}. Moreover, in \cite{Top2003} it was shown that given one pair of $PQ^{\epsilon}$-projective metrics, one can construct an infinite family of $PQ^{\epsilon}$-projective metrics. Under some non-degeneracy condition, this gives rise to an infinite family of integrable flows. 

From the other side, the theories of projectively equivalent and $h$-projectively equivalent metrics appeared to be very useful mathematical theories of deep interest.

The results in our paper suggest to look for other examples in the case when $\epsilon=-1,-3,-5,...$. If $\epsilon=-1$ but $P^{2}\neq-Id$, a lot of examples can be constructed using the "hierarchy construction" from \cite{Top2003}. It is interesting to ask whether every pair of $PQ^{-1}$-projective metrics is in the hierarchy of some $h$-projectively equivalent metrics.

Another attractive problem is to find interesting examples for $\epsilon=-3,-5,...$. Besides the relation to integrable systems provided by \cite{Top2003}, one could find other branches of differential geometry of similar interest as projective or $h$-projective geometry.

\subsection{PDE for $PQ^{\epsilon}$-projectivity}
Given a pair of Riemannian metrics $g,\bar{g}$ and tensors $P,Q$ satisfying \eqref{eq:conditions}, we introduce the $(1,1)$-tensor $A=A(g,\bar{g})$ defined by
\begin{align}
A=\left(\frac{\mathrm{det}\,\bar{g}}{\det{g}}\right)^{\frac{1}{m+1-\epsilon}}\bar{g}^{-1}g.\label{eq:defA}
\end{align}
Here we view the metrics as vector bundle isomorphisms $g:TM\rightarrow T^{*}M$ and $\bar{g}^{-1}:T^{*}M\rightarrow TM$. We see that $A$ is non-degenerate and self-adjoint with respect to $g$ and $\bar{g}$. Moreover $A$ commutes with $P$ and $Q$.
\begin{thm}[\cite{Top2003}]\label{thm:Top}
Two metrics $g$ and $\bar{g}$ are $PQ^{\epsilon}$-projective if for a certain vectorfield $\Lambda$, the $(1,1)$-tensor $A$ defined in \eqref{eq:defA} is a solution of 
\begin{align}
(\nabla_{X}A)Y=g(Y,X)\Lambda+g(Y,\Lambda)X+g(Y,QX)P\Lambda+g(Y,P\Lambda)QX\mbox{ for all }X,Y\in TM.\label{eq:main}
\end{align}
Conversely, if $A$ is a $g$-self-adjoint positive solution of \eqref{eq:main} which commutes with $P$ and $Q$, the Riemannian metric
$$\bar{g}=(\mathrm{det}\,A)^{-\frac{1}{1-\epsilon}}gA^{-1}$$
is $PQ^{\epsilon}$-projective to $g$.
\end{thm}
\begin{rem}\label{rem:lambda}
Taking the trace of the $(1,1)$-tensors in equation \eqref{eq:main} acting on the vector field $Y$, we obtain 
\begin{align}
\Lambda=\frac{1}{2(1-\epsilon)}\mathrm{grad}\,\mathrm{trace}\,A,\label{eq:lambda}
\end{align}
hence, \eqref{eq:main} is a linear first order PDE on the $(1,1)$-tensor $A$.
\end{rem}
\begin{rem}\label{rem:affine}
From Theorem \ref{thm:Top} it follows that the metrics $g,\bar{g}$ are affinely equivalent if and only if $\Lambda\equiv 0$ on the whole $M$. 
\end{rem}
\begin{rem}\label{rem:mainprojective}
The relation between the $1$-form $\Phi$ in \eqref{eq:pqproj} and the vectorfield $\Lambda$ in \eqref{eq:main} is given by $\Lambda=-Ag^{-1}\Phi$ (again $g^{-1}:T^{*}M\rightarrow TM$ is considered as a bundle isomorphism), see \cite{Top2003}. Recall from Example \ref{exmp:projective} that projective equivalence is a special case of $PQ^{\epsilon}$-projectivity with $\Phi(P.)=0$ or $Q=0$ and $\epsilon=0$. In view of Theorem \ref{thm:Top}, we now have that $g$ and $\bar{g}$ are projectively equivalent if and only if $A=A(g,\bar{g})$ given by \eqref{eq:defA} (with $\epsilon=0$), satisfies \eqref{eq:main} with $P\Lambda=0$ or $Q=0$, i.e.,
\begin{align}
(\nabla_{X}A)Y=g(Y,X)\Lambda+g(Y,\Lambda)X\mbox{ for all }X,Y\in TM.\label{eq:mainprojective}
\end{align}
\end{rem}

\section{Proof of the results}

\subsection{Topalov`s integrals}
We first recall
\begin{thm}[\cite{Top2003}]\label{thm:integrals}
Let $g$ and $\bar{g}$ be $PQ^{\epsilon}$-projective metrics and let $A$ be defined by \eqref{eq:defA}. We identify $TM$ with $T^{*}M$ by $g$, and consider the canonical symplectic structure on $TM\cong T^{*}M$. Then the functions $F_{t}:TM\rightarrow \mathbb{R}$, 
\begin{align}
F_{t}(X)=|\mathrm{det}\,(A-tId)|^{\frac{1}{1-\epsilon}}g((A-tId)^{-1}X,X),\,\,\,X\in TM\label{eq:integrals}
\end{align}
are commuting quadratic integrals for the geodesic flow of $g$.
\end{thm}
\begin{rem}
Note that the function $F_{t}$ in equation \eqref{eq:integrals} is not defined in the points $x\in M$ such that $t\in \mathrm{spec}\,A_{|x}$. From the proof of Theorem \ref{thm:result1} it will be clear that in the non-trivial case one can extend the functions $F_{t}$ to these points as well. 
\end{rem}

\subsection{Proof of Theorem \ref{thm:result1}}
Suppose that $g$ and $\bar{g}$ are $PQ^{\epsilon}$-projective Riemannian metrics and let $A=A(g,\bar{g})$ be the corresponding solution of \eqref{eq:main} defined by \eqref{eq:defA}.
Since $A$ is self-adjoint with respect to the positively-definite metric $g$, the eigenvalues of $A$ in every point $x\in M$ are real numbers. We denote them by $\mu_{1}(x)\leq...\leq\mu_{m}(x)$; depending on the multiplicity, some of the eigenvalues might coincide. The functions $\mu_{i}$ are continuous on $M$. Denote by $M^{0}\subseteq M$ the set of points where the number of different eigenvalues of $A$ is maximal on $M$. Since the functions $\mu_{i}$ are continuous, $M^{0}$ is open in $M$. Moreover, it was shown in \cite{Top2003} that $M^{0}$ is dense in $M$ as well. The implicit function theorem now implies that $\mu_{i}$ are differentiable functions on $M^{0}$. 

From Remark \ref{rem:affine} and equation \eqref{eq:lambda} we immediately obtain that $g$ and $\bar{g}$ are affinely equivalent, if and only if all eigenvalues of $A$ are constant. Suppose that $g$ and $\bar{g}$ are not affinely equivalent, that is, there is a non-constant eigenvalue $\rho$ of $A$ with multiplicity $k\geq1$. Let us choose a point $x_{0}\in M^{0}$ such that $d\rho_{|x_{0}}\neq0$, define $c:=\rho(x_{0})$ and consider the hypersurface $H=\{x\in U:\rho(x)=c\}$, where $U\subseteq M^{0}$ is a geodesically convex neighborhood of $x_{0}$. We think that $U$ is sufficiently small such that $\mu(x)\neq c$ for all eigenvalues $\mu$ of $A$ different from $\rho$ and all $x\in U$.
\begin{lem}\label{lem:tensor}
There is a smooth nowhere vanishing $(0,2)$-tensor $T$ on $U$ such that on $U\setminus H$, $T$ coincides with
\begin{align}
\mathrm{sgn}(\rho-c)|\mathrm{det}\,(A-cId)|^{\frac{1}{k}}g((A-cId)^{-1}.,.).\label{eq:tensor}
\end{align}
\end{lem}
\begin{proof}
Let us denote by $\rho=\rho_{1},\rho_{2},...,\rho_{r}$ the different eigenvalues of $A$ on $M^{0}$ with multiplicities $k=k_{1},k_{2},...,k_{r}$ respectively. Since the eigenspace distributions of $A$ are differentiable on $M^{0}$, we can choose a local frame $\{U_{1},...,U_{m}\}$ on $U$, such that $g$ and $A$ are given by the matrices 
$$g=\mathrm{diag}(1,...,1)\mbox{ and }A=\mathrm{diag}(\underbrace{\rho,...,\rho}_{k \mbox{ \tiny times}},...,\underbrace{\rho_{r},...,\rho_{r}}_{k_{r} \mbox{ \tiny times}})$$
with respect to this frame. The tensor \eqref{eq:tensor} can now be written as
\begin{align}
&\mathrm{sgn}(\rho-c)|\mathrm{det}\,(A-cId)|^{\frac{1}{k}}g(A-cId)^{-1}=\nonumber\\
&=(\rho-c)\prod_{i=2}^{r}|\rho_{i}-c|^{\frac{k_{i}}{k}}\mathrm{diag}\Big(\underbrace{\frac{1}{\rho-c},...,\frac{1}{\rho-c}}_{k \mbox{ \tiny times}},...,\underbrace{\frac{1}{\rho_{r}-c},...,\frac{1}{\rho_{r}-c}}_{k_{r} \mbox{ \tiny times}}\Big)=\nonumber\\
&=\prod_{i=2}^{r}|\rho_{i}-c|^{\frac{k_{i}}{k}}\mathrm{diag}\Big(\underbrace{1,...,1}_{k \mbox{ \tiny times}},...,\underbrace{\frac{\rho-c}{\rho_{r}-c},...,\frac{\rho-c}{\rho_{r}-c}}_{k_{r} \mbox{ \tiny times}}\Big).\label{eq:tensorframe}
\end{align}
Since $\rho_{i}\neq c$ on $U\subseteq M^{0}$ for $i=2,...,r$, we see that \eqref{eq:tensorframe} is a smooth nowhere vanishing $(0,2)$-tensor on $U$.
\end{proof}
\begin{lem}\label{lem:dim}
The multiplicity of the non-constant eigenvalues of $A$ is equal to $1-\epsilon$.
\end{lem}
\begin{proof}
Let us consider the integral $F_{c}:TM\rightarrow \mathbb{R}$ defined in equation \eqref{eq:integrals}. Using the tensor $T$ from Lemma \ref{lem:tensor}, we can write $F_{c}$ as 
\begin{align}
F_{c}(X)=\underbrace{\mathrm{sgn}(\rho-c)|\mathrm{det}\,(A-cId)|^{\frac{1}{1-\epsilon}-\frac{1}{k}}}_{=:f_{c}}T(X,X),\,\,\,X\in TM.\label{eq:integraltensor}
\end{align}
Our goal is to show that $\frac{1}{1-\epsilon}-\frac{1}{k}=0$. 
\begin{figure}
  \includegraphics[width=.3\textwidth]{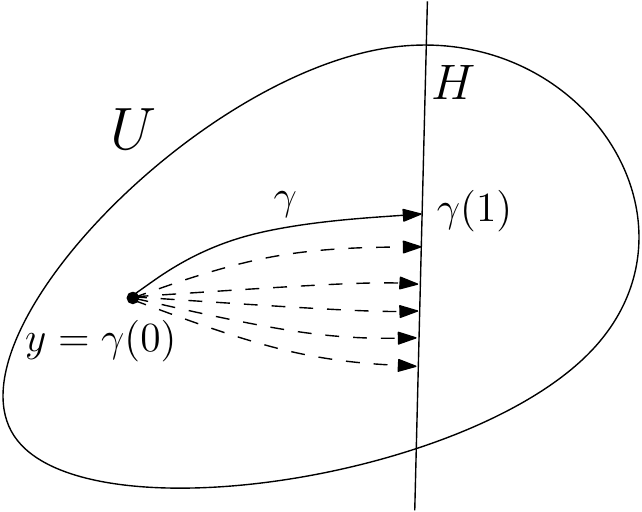}
  \caption{Case $\frac{1}{1-\epsilon}-\frac{1}{k}>0$: We connect the point $y\in U\setminus H$ with the points in $H$ by geodesics. The value of the integral $F_{c}$ is zero on each of these geodesics.}\label{pic1}
\end{figure}

First suppose that $\frac{1}{1-\epsilon}-\frac{1}{k}>0$ and let be $y\in U\setminus H$. We choose a geodesic $\gamma:[0,1]\rightarrow U$ such that $y=\gamma(0)$ and $\gamma(1)\in H$, see figure \ref{pic1}. Since $\rho(\gamma(t))\stackrel{t\to 1}{\longrightarrow}c$, we see from equation \eqref{eq:integraltensor} that $f_{c}(\gamma(t))\stackrel{t\to 1}{\longrightarrow} 0$.
It follows that $F_{c}(\dot{\gamma}(t))\stackrel{t\to 1}{\longrightarrow}0$. On the other hand, since $F_{c}$ is an integral for the geodesic flow of $g$ (see Theorem \ref{thm:Top}), the value $F_{c}(\dot{\gamma}(t))$ is independent of $t$ and, hence, $F_{c}(\dot{\gamma}(0))=0$. We have shown that $F_{c}(\dot{\gamma}(0))=0$ for all initial velocities $\dot{\gamma}(0)\in T_{y}M$ of geodesics connecting $y$ with points of $H$.
Since $H$ is a hypersurface, it follows that the quadric $\{X\in T_{y}M:F_{c}(X)=0\}$ contains an open subset which implies that $F_{c}\equiv 0$ on $T_{y}M$. This is a contradiction to Lemma \ref{lem:tensor}, since $T$ is non-vanishing in $y$. We obtain that $\frac{1}{1-\epsilon}-\frac{1}{k}\leq0$.
\begin{figure}
  \includegraphics[width=.3\textwidth]{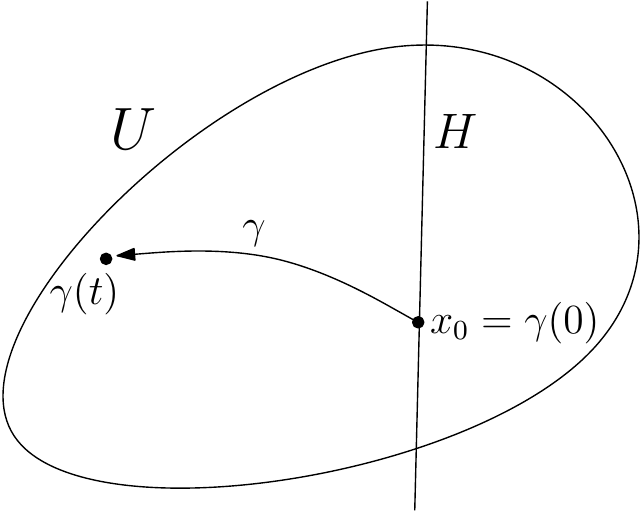}
  \caption{Case $\frac{1}{1-\epsilon}-\frac{1}{k}<0$: For any geodesic $\gamma$ starting in $x_{0}\in H$ and leaving $H$, the value of the integral $F_{c}$ along this geodesic is infinite.}\label{pic2}
\end{figure}

Let us now treat the case when $\frac{1}{1-\epsilon}-\frac{1}{k}<0$. We choose a vector $X\in T_{x_{0}}M$ which is not tangent to $H$ and satisfies $T(X,X)\neq0$. Such a vector exists, since $T_{x_{0}}M\setminus T_{x_{0}}H$ is open in $T_{x_{0}}M$ and $T$ is not identically zero on $T_{x_{0}}M$ by Lemma \ref{lem:tensor}. 
Let us consider the geodesic $\gamma$ with $\gamma(0)=x_{0}$ and $\dot{\gamma}(0)=X$, see figure \ref{pic2}. Since $X\notin T_{x_{0}}H$, the geodesic $\gamma$ has to leave $H$ for $t>0$. In a point $\gamma(t)\in U\setminus H$ the value $F_{c}(\dot{\gamma}(t))$ will be finite. On the other hand, since $f_{c}(\gamma(t))\stackrel{t\to 0}{\longrightarrow} \infty$ and $T(\dot{\gamma}(0),\dot{\gamma}(0))\neq0$, we have $F_{c}(\dot{\gamma}(t))\stackrel{t\to 0}{\longrightarrow} \infty$. Again this contradicts the fact that the value of $F_{c}$ must remain constant along $\dot{\gamma}$ by Theorem \ref{thm:Top}. We have shown that $\frac{1}{1-\epsilon}-\frac{1}{k}=0$ and finally, Lemma \ref{lem:dim} is proven.
\end{proof}
As a consequence of Lemma \ref{lem:dim}, if the metrics $g,\bar{g}$ are not affinely equivalent (i.e., at least one eigenvalue of $A$ is non-constant), $\epsilon$ is an integer less or equal to zero. If $\epsilon\neq0$, the condition $PQ=\epsilon Id$ in \eqref{eq:conditions} implies that $P$ is non-degenerate and by the first condition in \eqref{eq:conditions}, $g(P.,.)$ is a non-degenerate $2$-form on each eigenspace of $A$ (note that $A$ and $P$ commute). This implies that for $\epsilon\neq0$ the eigenspaces of $A$ have even dimension, in particular, $1-\epsilon\in\{2,4,6,8,...\}$. Theorem \ref{thm:result1} is proven.   
\subsection{Proof of Theorem \ref{thm:result2}}
Let $g,\bar{g}$ be two $PQ^{\epsilon}$-projective metrics and let $A$ be the corresponding solution of equation \eqref{eq:main} defined by \eqref{eq:defA}. 
As it was already stated in the proof of Theorem \ref{thm:result1}, the eigenspace distributions of $A$ are differentiable in a neighborhood of almost every point of $M$. First let us prove 
\begin{lem}\label{lem:eigenvectors}
Let $X$ be an eigenvector of $A$ corresponding to the eigenvalue $\rho$. If $\mu$ is another eigenvalue of $A$ and $\rho\neq\mu$, then $X(\mu)=0$. In particular, $\mathrm{grad}\,\mu$ is an eigenvector of $A$ corresponding to the eigenvalue $\mu$.
\end{lem}
\begin{rem}
Lemma \ref{lem:eigenvectors} is known for projectively equivalent (Example \ref{exmp:projective}) and $h$-projectively equivalent (Example \ref{exmp:hprojective}) metrics. For projectively equivalent metrics it is a classical result which was already known to Levi-Civita \cite{Levi1896}. For $h$-projectively equivalent metrics, it follows from \cite{ApostolovI,MatRos}.  
\end{rem}
\begin{proof}
Let $Y$ be an eigenvector field of $A$ corresponding to the eigenvalue $\mu$. For arbitrary $X\in TM$, we obtain $\nabla_{X}(AY)=\nabla_{X}(\mu Y)=X(\mu)Y+\mu \nabla_{X}Y$ and $\nabla_{X}(AY)=(\nabla_{X}A)Y+A\nabla_{X}Y$. Combining these equations and replacing the expression $(\nabla_{X}A)Y$ by \eqref{eq:main} we obtain
\begin{align}
(A-\mu Id)\nabla_{X}Y=X(\mu)Y-g(Y,X)\Lambda-g(Y,\Lambda)X-g(Y,QX)P\Lambda-g(Y,P\Lambda)QX.\label{eq:covderiv}
\end{align}
Now let $X$ be an eigenvector of $A$ corresponding to the eigenvalue $\rho$ and suppose that $\rho\neq\mu$. Since $A$ is $g$-self-adjoint, the eigenspaces of $A$ corresponding to different eigenvalues are orthogonal to each other. Moreover, since $A$ and $Q$ commute, $Q$ leaves the eigenspaces of $A$ invariant. Using \eqref{eq:covderiv} we obtain 
\begin{align}
(A-\mu Id)\nabla_{X}Y+g(Y,\Lambda)X+g(Y,P\Lambda)QX=X(\mu)Y.\nonumber
\end{align}
Since the left-hand side is orthogonal to the $\mu$-eigenspace of $A$, we necessarily have $X(\mu)=0$. We have shown that $g(\mathrm{grad}\,\mu,X)=X(\mu)=0$ for any eigenvalue $\mu$ and any eigenvector field $X$ corresponding to an eigenvalue different form $\mu$. This forces $\mathrm{grad}\,\mu$ to be contained in the eigenspace of $A$ corresponding to $\mu$. 
\end{proof}
Now suppose that $\epsilon=0$. Let us denote the non-constant eigenvalues of $A$ by $\rho_{1},...,\rho_{l}$. Using Lemma \ref{lem:dim}, the corresponding eigenspaces are $1$-dimensional and Lemma \ref{lem:eigenvectors} implies that they are spanned by the gradients $\mathrm{grad}\,\rho_{1},...,\mathrm{grad}\,\rho_{l}$ respectively. Since $P$ and $A$ commute, $P$ leaves the eigenspaces of $A$ invariant, hence, $P\mathrm{grad}\,\rho_{i}=p_{i}\mathrm{grad}\,\rho_{i}$ for some real number $p_{i}$. Now $P$ is skew with respect to $g$ and we obtain $0=g(\mathrm{grad}\,\rho_{i},P\mathrm{grad}\,\rho_{i})=p_{i}g(\mathrm{grad}\,\rho_{i},\mathrm{grad}\,\rho_{i})$ which implies that 
$$P\mathrm{grad}\,\rho_{i}=0.$$
On the other hand, by equation \eqref{eq:lambda} 
$$\Lambda=\frac{1}{2}\mathrm{grad}\,\mathrm{trace}\,A=\frac{1}{2}(\mathrm{grad}\,\rho_{1}+...+\mathrm{grad}\,\rho_{l}).$$ 
Combining the last two equations, we obtain $P\Lambda=0$. It follows from Remark \ref{rem:mainprojective} that $g$ and $\bar{g}$ are projectively equivalent and, hence, Theorem \ref{thm:result2} is proven.

\nocite{*}
\bibliographystyle{plain}

\end{document}